\documentclass[11pt]{article}

\makeatletter
\@addtoreset{footnote}{page}
\makeatother

\usepackage{style-enumitem}

\usepackage{amsthm}
\usepackage{amsmath,amssymb,latexsym,amsfonts,mathrsfs}

\newcommand{\n}{\nonumber}

\renewcommand{\hat}{\widehat}

\renewcommand{\bar}{\overline}

\newcommand{\rbra}[1]{\!\left( #1 \right)} 
\newcommand{\cbra}[1]{\!\left\{ #1 \right\}} 
\newcommand{\sbra}[1]{\!\left[ #1 \right]} 

\newcommand{\bE}{\ensuremath{\mathbb{E}}}

\newcommand{\bR}{\ensuremath{\mathbb{R}}}

\theoremstyle{plain}

\theoremstyle{definition}

\numberwithin{equation}{section}

\makeatletter
\renewcommand\section{\@startsection {section}{1}{\z@}%
                                   {-3.5ex \@plus -1ex \@minus -.2ex}%
                                   {2.3ex \@plus.2ex}%
                                   {\normalfont\large\bf}}
\makeatother

\makeatletter
\renewcommand\subsection{\@startsection {subsection}{1}{\z@}%
                                   {-3.5ex \@plus -1ex \@minus -.2ex}%
                                   {2.3ex \@plus.2ex}%
                                   {\normalfont\normalsize\bf}}
\makeatother

\usepackage{amsmath,amsfonts,amssymb,graphicx,amsthm, times}
\usepackage{latexsym,delarray,verbatim}
\usepackage{epsfig,color,ifthen}

\newcommand {\ud}{{\rm d}}

\numberwithin{equation}{section}

\newtheorem{theorem}{Theorem}[section]
\newtheorem{proposition}{Proposition}[section]

\newtheorem{remark}{Remark}[section]
\newtheorem{lemma}{Lemma}[section]

\numberwithin{remark}{section} \numberwithin{proposition}{section}
\numberwithin{corollary}{section}
\newcommand {\R}{\mathbb{R}}

\newcommand {\p}{\mathbb{P}}

\newcommand {\E}{\mathbb{E}}

\newcommand{\diff}{{\rm d}}

\newcommand{\lev}{L\'{e}vy }

\newcommand{\e}{\mathbb{E}}

\begin{document}

\begin{center}
{\Large \bf 
On optimal periodic dividend and capital injection strategies for spectrally negative \lev models
}
\end{center}
\begin{center}
Kei Noba, Jos\'e-Luis P\'erez, Kazutoshi Yamazaki and Kouji Yano
\end{center}
\begin{center}
{\small \today}
\end{center}

\begin{abstract}
De Finetti's optimal dividend problem has recently been extended to the case dividend payments can only be made at Poisson arrival times. This paper considers the version with bail-outs where the surplus must be nonnegative uniformly in time.  For a general spectrally negative \lev model, we show the optimality of a Parisian-classical reflection strategy that pays the excess above a given barrier at each Poisson arrival times and  also reflects from below at zero in the classical sense.

\vspace{0.1cm}
\noindent \small{\noindent  AMS 2010 Subject Classifications: 60G51, 93E20, 91B30
}
\end{abstract}


\section{Introduction}

In the bail-out model of de Finetti's dividend problem, a joint optimal dividend and capital injection strategy is pursued so as to maximize the total expected dividend payments  minus the costs of capital injections.  In the past decade, the classical Cram\'er-Lundberg model has been generalized to a spectrally negative \lev model.  In particular, Avram et al.\  \cite{AvrPalPis} showed the optimality of a double barrier strategy that reflects from below at zero and also from above at a certain barrier.

In this paper, we consider its extension with a periodic dividend constraint.  Periodic observation models have recently been studied widely in the insurance literature (as in, e.g., Albrecher et al.\ \cite{ABT, ACT}).  For the case without capital injections in which dividends are paid until the time of ruin, a periodic barrier strategy that pays any excess above a certain barrier at each payment opportunity is expected to be optimal.  Its optimality has been confirmed for the spectrally positive \lev (dual) models by Avanzi et al.\ \cite{ATW} and P\'erez and Yamazaki \cite{PerYam}, and for the spectrally negative \lev models with a completely monotone \lev density by 
 Noba et al.\ \cite{NobPerYamYan}.  On the other hand, regarding the bail-out case,  the optimality results are only available for the dual model given in the second problem considered in \cite{PerYam}, to the best of our knowledge.

The objective of this paper is to show the optimality of a \emph{periodic-classical barrier strategy} 
under a general spectrally negative \lev model.  This can be seen as the bail-out extension of \cite{NobPerYamYan}  and also as the spectrally negative version of the bail-out model in \cite{PerYam}.  

We follow the \emph{guess and verify} procedure to tackle the problem.  Under a periodic-classical barrier strategy, the expected NPVs of dividends and capital injections can be written in terms of the scale function by the results given in \cite{PerYam2}.  The candidate optimal barrier is first chosen using the conjecture that the slope of the value function at the barrier becomes one. The optimality of the selected strategy is then confirmed by showing that the candidate value function solves the proper variational inequalities. This is indeed satisfied by the convexity of the candidate value function, that is shown by our observation that its slope becomes proportional to a certain ruin identity, 
which is monotone in the starting value of the process.

Regarding the comparison with the dual model \cite{PerYam}, there are both similarities and differences.  In this paper, we focus on the differences and omit similar results, such as the verification lemma, that can be attained similarly to \cite{PerYam}. 
In the dual model, only minimal modifications are necessary to solve the bail-out case from the case with ruin.  As shown in \cite{PerYam}, the value functions for both cases admit exactly the same expressions except that the optimal barriers are different.  
On the other hand, this is not expected in the spectrally negative \lev model.  The expressions of the optimal solutions are different, and we use different approaches to show the variational inequalities.
It is noted that in this paper we \emph{do not assume} the completely monotone density assumption which was needed in  \cite{NobPerYamYan}.



The rest of the paper is organized as follows.  The considered problem is formulated and  a review of the spectrally negative L\'evy process is given in Section \ref{section_preliminaries}. In Section \ref{PR}, we define the periodic-classical barrier strategies and construct the corresponding surplus process. We also provide a review of the scale function and obtain the expected NPVs corresponding to these strategies. In Section \ref{sf_d} we obtain the optimal barrier for the periodic-classical strategy, and in Section \ref{VoO} we prove that the expected NPVs associated with this strategy solves the proper variational inequalities. Finally, in Section \ref{section_numerics}, we provide some numerical results.

\section{Preliminaries} \label{section_preliminaries}
\subsection{Spectrally negative \lev processes} 
Let $X=(X(t); t\geq 0)$ be defined on a  probability space $(\Omega, \mathcal{F}, \p)$, modeling the surplus of an insurance company in the absence of control.  For $x\in \R$, we denote by $\p_x$ the law of $X$ when it starts at $x$ and write for convenience  $\p$ in place of $\p_0$. Accordingly, we shall write $\e_x$ and $\e$ for the associated expectation operators. 

In this paper, we assume that $X$ is a spectrally negative \lev process that is not the negative of a subordinator, and its Laplace exponent  $\psi(\theta):[0,\infty) \to \R$ is such that
\[
\e\big[{\rm e}^{\theta X(t)}\big]=:{\rm e}^{\psi(\theta)t}, \qquad t, \theta\ge 0,
\]
given by the \emph{L\'evy-Khintchine formula} 
\begin{equation}\label{lk}
	\psi(\theta):=\gamma\theta+\frac{\eta^2}{2}\theta^2+\int_{(-\infty,0)}\big({\rm e}^{\theta z}-1-\theta z \mathbf{1}_{\{z >-1\}}\big)\Pi(\ud z), \quad \theta \geq 0.
\end{equation}
Here, $\gamma\in \R$, $\eta \ge 0$, and $\Pi$ is a \lev measure on $(-\infty,0)$ such that
\[
\int_{(-\infty,0)}(1\land z^2)\Pi(\ud z)<\infty.
\]

The process $X$ has paths of bounded variation if and only if $\eta=0$ and $\int_{(-1, 0)} |z|\Pi(\mathrm{d}z)$ is finite. In this case $X$ can be written as
\begin{equation}
	X(t)=ct-S(t), \,\,\qquad t\geq 0,\notag
\end{equation}
where 
\begin{align*}
	c:=\gamma-\int_{(-1,0)} z\Pi(\mathrm{d}z) 
\end{align*}
and $(S(t); t\geq0)$ is a driftless subordinator. By the assumption that it does not have monotone paths, we must have  $c>0$ and we can write
\begin{equation*}
	\psi(\theta) = c \theta+\int_{(-\infty,0)}\big( {\rm e}^{\theta z}-1\big)\Pi(\ud z), \quad \theta \geq 0. 
\end{equation*}

\subsection{
The optimal Poissonian dividend problem with classical capital injection}\label{strategy} 

A (dividend/capital injection) strategy  is a pair of processes $\pi := \left( L^{\pi}(t), R^{\pi}(t); t \geq 0 \right)$  consisting of the cumulative amount of dividends $L^{\pi}$ and those of capital injection $R^{\pi}$.
\par 
Regarding the dividend strategy, we assume that the dividend payments can only be made at the arrival times $\mathcal{T}_r :=(T(i); i\geq 1 )$ of a Poisson process $N^r=( N^r(t); t\geq 0) $ with intensity $r>0$, which is independent of the L\'evy process $X$.  
In other words, $T(i)-T(i-1)$, $i\geq 1$ (with $T(0) := 0$) are independent and exponentially distributed with mean $1/r$.  
More precisely,  $L^{\pi}$ admits the form
\begin{align}
L^{\pi}(t)=\int_{[0,t]}\nu^{\pi}(s)\diff N^r(s),\qquad\text{$t\geq0$,} \label{restriction_poisson}
\end{align}
for some c\`agl\`ad process $\nu^{\pi}$ adapted to the filtration $\mathbb{F} := (\mathcal{F}(t); t \geq 0)$ generated by the processes $(X, N^r)$.

Regarding the capital injection, we assume that $R^{\pi}$ is a  nondecreasing, right-continuous, and $\mathbb{F}$-adapted process with $R^{\pi}(0-) = 0$. Contrary to the dividend payments, capital injection can be made continuously. 

The corresponding risk process is given by $U^{\pi}(0-) = X(0)$ and 
\begin{align*}
	U^{\pi}(t) := X(t) - L^{\pi}(t) + R^{\pi}(t), \quad t \geq 0,
\end{align*}
and $(L^{\pi}, R^{\pi})$ must be chosen so that $U^{\pi}(t) \geq 0$ for all $t \geq 0$ a.s.


Assuming that $\beta > 1$ is the cost per unit injected capital and $q > 0$ is the discount factor, the objective is to maximize
\begin{align} \label{v_pi}
v_{\pi} (x) := \mathbb{E}_x \left( \int_{[0, \infty)} e^{-q t} \diff L^{\pi}(t) - \beta \int_{[0, \infty)} e^{-q t} \diff R^{\pi}(t)\right), \quad x\geq 0,
\end{align}
over  the set of all admissible strategies $\mathcal{A}$ that satisfy all the constraints described above and 
\begin{align}
\E_x\left(\int_{[0, \infty)} e^{-qt} \diff R^{\pi}(t)\right) < \infty. 
\label{admissibility2}
\end{align}
Hence the problem is to compute the value function
\begin{equation}\label{control:value}
v(x):=\sup_{\pi \in \mathcal{A}}v_{\pi}(x), \quad x \geq 0,
\end{equation}
and obtain an optimal strategy $\pi^*$ that attains it, if such a strategy exists.
Throughout the paper, for the solution to be nontrivial, we assume
\begin{align}\label{mean-finite}
	\E [X(1)] =  \psi'(0+) >- \infty.
\end{align}

\section{
	Periodic-classical barrier strategies}\label{PR}
As in the spectrally positive case \cite{PerYam}, the objective of this paper is to
show the optimality of the periodic-classical barrier strategy 
\[
\bar{\pi}^{0,b} := \{(L_r^{0,b}(t), R_r^{0,b}(t)); t \geq 0 \}.
\]
The controlled process $U_r^{0,b}$ becomes the \textit{L\'evy process with Parisian reflection above and classical reflection below} considered in \cite{PerYam}, which can be constructed as follows.

Let  $R (t):= (-\inf_{0 \leq s \leq t} X(s)) \vee 0$ for $t\geq 0$, then we have 
\begin{align*}
U_r^{0,b}(t) = X(t) + R(t), \quad 0 \leq t < \widehat{T}_b^{+} (1) 
\end{align*}
where $\widehat{T}_b^{+}(1) := \inf\{T(i):\; X(T(i))+R(T(i)) > b\}$.
The process then jumps down by $
X(\widehat{T}_b^{+}(1))+R(\widehat{T}_b^{+}(1))-b$ so that $U_r^{0,b}(\widehat{T}_b^{+}(1)) = b$. For $\widehat{T}_b^{+}(1) \leq t < \widehat{T}_b^{+}(2)  := \inf\{T(i) > \widehat{T}_b^{+}(1):\; U_r^{0,b}(T(i) -) > b\}$, $U_r^{0,b}(t)$ is the process reflected  at $0$ of  the process $( X(t) - X(\widehat{T}_b^{+}(1)) +b; t \geq \widehat{T}_b^+(1) )$. 
The process $U_r^{0,b}$ can be constructed by repeating this procedure.
It is clear that it admits a decomposition
\begin{align*}
U_r^{0,b}(t) = X(t) - L_r^{0,b}(t) + R_r^{0,b}(t), \quad t \geq 0,
\end{align*}
where $L_r^{0,b}(t)$ and $R_r^{0,b}(t)$ are, respectively, the cumulative amounts of Parisian and classical reflection until time $t$. 

We will see that the strategy $\bar{\pi}^{0,b} := \{(L_r^{0,b}(t), R_r^{0,b}(t)); t \geq 0 \}$, for $b \geq 0$, is admissible for the  problem described in Section \ref{strategy} (because \eqref{admissibility2} holds by Lemma \ref{lemma_NPV} and our assumption \eqref{mean-finite}). Its expected NPV of dividends minus the costs of capital injection is denoted by 
\begin{align} \label{v_pi}
v_b(x) := \mathbb{E}_x \left( \int_{[0, \infty)} e^{-q t} \diff L_r^{0,b}(t) - \beta \int_{[0, \infty)} e^{-q t} \diff R^{0,b}_r(t)\right), \quad x\geq 0.
\end{align}


\subsection{Scale functions} 
For fixed $q \geq 0$, let $W^{(q)}: \R \to [0 , \infty)$ be the scale function of the spectrally negative 
L\'evy process $X$. 
This takes the value zero on the 
negative half-line, and on the positive half-line it is a continuous and strictly 
increasing function  defined by its Laplace transform: 
\begin{align}\label{scale_function_laplace}
\int_0^\infty e^{-\theta x} W^{(q)} (x) \diff x 
= \frac{1}{\psi (\theta ) -q} , ~~\theta > \Phi (q), 
\end{align}
where $\psi$ is as defined in \eqref{lk} and
\begin{align}
\begin{split}
\Phi(q) := \sup \{ \lambda \geq 0: \psi(\lambda) = q\} . 
\end{split}
\label{def_varphi}
\end{align}
We also define, for all $x \in \R$, 
\begin{align*}
&{\bar{W}}^{(q)} (x) := \int_0^x W^{(q)} (y) \diff y,  ~~~
{\bar{\bar{W}}}^{(q)} (x) : = \int_0^x \int_0^z W^{(q)} (w)\diff w \diff z, \\
&Z^{(q)} (x) := 1+ q\bar{W}^{(q)}(x), ~~~
\bar{Z}^{(q)} (x) := \int_0^x Z^{(q)} (z) \diff z=x +q\bar{\bar{W}}^{(q)} (x) . 
\end{align*}
		Because $W^{(q)}(x) = 0$ for $-\infty < x < 0$, we have
		\begin{align*}
		\overline{W}^{(q)}(x) = 0,\quad \overline{\overline{W}}^{(q)}(x) = 0,\quad Z^{(q)}(x) = 1,
		\quad \textrm{and} \quad \overline{Z}^{(q)}(x) = x, \quad x \leq 0.  
		\end{align*}
		
		

			\begin{remark} \label{remark_scale_function_properties}
				\begin{enumerate}
					\item[(1)] $W^{(q)}$ is differentiable a.e. In particular, if $X$ is of unbounded variation or the \lev measure does not have an atom, it is known that $W^{(q)}$ is $C^1(\R \backslash \{0\})$; see, e.g.,\ \cite[Theorem 3]{Chan2011}.
					\item[(2)] As in Lemma 3.1 of \cite{KKR},
					\begin{align*}
					\begin{split}
					W^{(q)} (0) &= \left\{ \begin{array}{ll} 0 & \textrm{if $X$ is of unbounded
						variation,} \\ c^{-1} & \textrm{if $X$ is of bounded variation.}
					\end{array} \right.  
					\end{split}
					\end{align*}
				\end{enumerate}
			\end{remark}

		We also use 
$W^{(q+r)}$ and $\Phi(q+r)$, which are defined by \eqref{scale_function_laplace} and \eqref{def_varphi} with $q$ replaced by $q+r$. 
By the convexity of $\psi$ on $(0,\infty)$,
		we have $\Phi(q+r) > \Phi(q)$ for $r > 0$, 
		and, from the identity (5) in \cite{LRZ}, 
		\begin{align*}
		W^{(q+r)}(x)-W^{(q)}(x)=r\int_0^xW^{(q+r)}(u)W^{(q)}(x-u) \diff u, \quad x \in \R. 
		\end{align*}
		
We also define, for $q, r > 0$ and $x \in \R$,
		\begin{align*}
		Z^{(q)}(x,\Phi(q+r)) &:=e^{\Phi(q+r) x} \left( 1 -r \int_0^{x} e^{-\Phi(q+r) z} W^{(q)}(z) \diff z	\right) \notag \\
		&=r \int_0^{\infty} e^{-\Phi(q+r) z} W^{(q)}(z+x) \diff z	 > 0. 
\end{align*}
		Here, the second equality holds because \eqref{scale_function_laplace} gives $\int_0^\infty  \mathrm{e}^{-\Phi(q+r) x} W^{(q)}(x) \diff x = r^{-1}$.
		By differentiating this with respect to the first argument,
		\begin{align*}
		Z^{(q) \prime}(x,\Phi(q+r)) &:= \frac \partial {\partial x}Z^{(q)}(x,\Phi(q+r))  = \Phi(q+r) Z^{(q)}(x,\Phi(q+r))	- r W^{(q)}(x), \quad x > 0. 
		\end{align*}

Finally, for $b\geq 0$ and $x \in \bR$,  we define
\begin{align}
\begin{split}
W_{-b}^{(q, r)} (x ) 
&:= W^{(q)}(x+ b) + r \int_0^{x } W^{(q + r)} (y)W^{(q)} (x - y +b) \diff y,\\
Z_{-b}^{(q, r)} (x ) 
&:= Z^{(q)}(x+ b) + r \int_0^{x } W^{(q + r)} (y) Z^{(q)} (x - y +b) \diff y, 
\\
\overline{Z}_{-b}^{(q, r)} (x ) 
&:= \overline{Z}^{(q)}(x+ b) + r \int_0^{x } W^{(q + r)} (y) \overline{Z}^{(q)} (x - y +b) \diff y. 
\end{split}
\label{identities} 
\end{align}
\begin{remark}\label{identities_LRZ}
	Using the identities given in (5) of \cite{LRZ}, we have
	\begin{align*}
	W_{0}^{(q, r)} (x)=W^{(q+r)}(x) \quad\text{and}\quad Z_{0}^{(q, r)} (x)=Z^{(q+r)}(x), \quad x \in \R.
	\end{align*}
\end{remark}

\begin{remark}
Fix $b \geq 0$. Let
	$X_{r}$ be the Parisian reflected process of $X$ from above at the level $0$ (without classical reflection) as studied in \cite{PerYam2}, and 
\[
\tau_{-b}^-(r):=\inf\{t>0: X_r(t)<-b \}.
\]
Then, by Corollary 3.3 in \cite{PerYam2}, for any $x\in\mathbb{R}$, 
\begin{align}\label{403}
\mathbb{E}_{x-b}&\Big[e^{-\tau_{-b}^-(r)}\Big]
=Z_{-b}^{(q,r)}(x-b)-rZ^{(q)}(b)\overline{W}^{(q+r)}(x-b)\notag\\&-q\frac{Z^{(q)}(b,\Phi(q+r))}{Z^{(q)\prime}(b,\Phi(q+r))}\left(W^{(q,r)}_{-b}(x-b)-rW^{(q)}(b) \overline{W}^{(q+r)}(x-b)\right),
\end{align}
where in particular
\begin{align}\label{fpt_pr}
\mathbb{E}_0\left[e^{-\tau_{-b}^-(r)}\right]
&=Z^{(q)}(b) -q\frac{Z^{(q)}(b,\Phi(q+r))}{Z^{(q)\prime}(b,\Phi(q+r))}W^{(q)}(b).\end{align}
These identities are used later in Remark \ref{remark_g_probabilistic} and the proof of Lemma \ref{cond_b^*}.
\end{remark}

For a comprehensive study on the scale function and its applications, see \cite{KKR,Kyp}.

\subsection{Expression of $v_b$ via the scale function} 




Using the functions given in \eqref{identities}, the expression \eqref{v_pi} can be computed immediately by Corollaries 4.4 and 4.5 in \cite{PerYam2} via the scale function. Recall our assumption \eqref{admissibility2} that $\psi'(0+)$ is finite.
Below,  we extend the domain of $v_b$ to $\R$ by setting $v_b(x) = \beta x + v_b(0)$ for $x < 0$, so as to include the case when the process is started at a negative value and is pushed up to zero immediately.

\begin{lemma}\label{lemma_NPV}
For $b\geq0$ and $x \in \R$,
\begin{align}\label{vf_new_label}
v_b(x)&=-C_{b}\left(Z_{-b}^{(q,r)}(x-b)-rZ^{(q)}(b)\overline{W}^{(q+r)}(x-b)\right)-r\overline{\overline{W}}^{(q+r)}(x-b)\notag\\
&+\beta\Big( 
\overline{Z}_{-b}^{(q, r)} (x -b)
+\frac{\psi^\prime(0+)}{ q}  -r \overline{Z}^{(q)}(b)\overline{W}^{(q+r)}(x-b)
\Big),
\end{align}
where 
\begin{equation}\label{C_b}
C_{b}:=\frac{r(\beta Z^{(q)} (b)-1)}{q\Phi(q+r)Z^{(q)}(b,\Phi(q+r))} + \frac {\beta} {\Phi(q+r)}. 
\end{equation}
In particular, for $x\leq b$, we obtain that
\begin{align}\label{vf_new_label_x<b}
v_{b}(x)
=-C_{b}Z^{(q)}(x)+\beta \Big( \bar{Z}^{(q)} (x)+\frac{\psi^\prime(0+)}{ q}\Big).
\end{align}
\end{lemma}
\begin{proof}
By Corollaries 4.4 and 4.5 in \cite{PerYam2}, for all $b \geq 0$ and $x \in \R$, we have
\begin{align*}
\bE_x \rbra{\int_{[0 , \infty)} e^{-qt} \diff L_r^{0 , b} (t)} &= 
r\rbra{ \frac{Z_{-b}^{(q, r)} (x-b) - r Z^{(q)}(b) \bar{W}^{(q + r)} (x-b)}{q\Phi (q +r)Z^{(q)}(b, \Phi (q+r)) }
 -\bar{\bar{W}}^{(q+r)} (x - b)}, 
 \\
\bE_x \rbra{\int_{[0 , \infty)} e^{-qt} \diff R_r^{0 , b} (t)}
&=\rbra{\frac{rZ^{(q)} (b) }{ q \Phi (q + r)Z^{(q)}(b , \Phi (q + r) )} + \frac 1 {\Phi(q+r)}} \notag \\ & \times \Big( Z_{-b}^{(q, r)} (x-b) - r Z^{(q)}(b) \bar{W}^{(q + r)} (x-b) \Big) \notag\\
&- \Big( \overline{Z}_{-b}^{(q, r)} (x-b) +\frac{\psi^\prime(0+)}{ q}  - r \bar{Z}^{(q)} (b)   \bar{W}^{(q + r)} (x-b) \Big). 
\end{align*}
Combining these, we have the claim.
\end{proof}






	\subsection{Smoothness of $v_b$}
	
	Here we  analyze the smoothness of the function $v_{b}$. 
	 The proof of the following lemma is straightforward and is hence omitted. 
	\begin{lemma}\label{derivatives_v}
		For all $b \geq 0$, 
		\begin{align}
		v_b^{\prime} (x) &= -qC_{b} W^{(q,r)}_{-b}(x-b)-r\overline{W}^{(q+r)}(x-b)+\beta Z^{(q,r)}_{-b}(x-b), \quad x \in \R \backslash \{ 0\}, 
\label{302a}\\
		v_b^{ \prime \prime} (x +) &=  -qC_{b} \left(W^{(q)\prime}(x +)+rW^{(q+r)}(x-b)W^{(q)}(b)+r\int_0^{x-b}W^{(q+r)}(y)W^{(q)\prime}(x-y)\diff y \right)\notag\\&-rW^{(q+r)}(x-b)
		+\beta\Big(q W^{(q,r)}_{-b}(x-b)+rW^{(q+r)}(x-b)Z^{(q)}(b)\Big), \quad x \in \R \backslash \{0,b\}. 
		\label{303a}
		\end{align} 
	\end{lemma}
	
	By the smoothness of the scale function on $\R \backslash \{0\}$ as in Remark \ref{remark_scale_function_properties}(1), the derivative \eqref{302a} 
	 is continuous on  $\R \backslash \{0\}$. In particular, in the case  $X$ is of unbounded variation, by Remark \ref{remark_scale_function_properties} (1) and (2), the second derivative, given by \eqref{303a}, 
	  is continuous on $\R \backslash \{0\}$.  
	Hence, we have the following results.
	\begin{lemma} \label{lemma_smooth_fit} For all $b \geq 0$, we have the following:
		\begin{enumerate}
			\item When $X$ is of bounded variation, $v_b$ is continuously differentiable on $\R \backslash \{0\}$.
			\item When $X$ is of unbounded variation, $v_b$ is twice continuously differentiable on $\R \backslash \{0\}$.
		\end{enumerate}
	\end{lemma}
	
	\begin{remark}[Continuity/smoothness at zero] \label{remark_smoothness_zero} For $b \geq 0$, we have the following. 
	\begin{enumerate}
		\item By Lemma \ref{lemma_NPV}, we have that $v_{b}$ is continuous at zero.
		\item  For the case $X$ is of unbounded variation,  $v_{b}$ is continuously differentiable at zero because, 
by Lemma \ref{derivatives_v} and Remark \ref{remark_scale_function_properties}(2),
		$v_{b}'(0+)
		= -qC_{b} W^{(q)}(0) + \beta  = \beta = v_{b}'(0-)$.
		\end{enumerate}
	\end{remark}
	\section{Selection of a candidate optimal barrier $b^*$}\label{sf_d} 
In this section, we focus on the periodic barrier strategy defined in the previous section and choose the candidate barrier $b^*$, 
which satisfies
$v_{b^*}'(b^*)=1$
if such $b^* > 0$ exists, and set it to be zero otherwise.

Recall, as in Lemma \ref{lemma_smooth_fit}, that $v_b$ is continuously differentiable except at zero. If $b >0$, using \eqref{C_b} and \eqref{302a}, 
we have 
\begin{align} \label{slope_b}
\begin{split}
v_b'(b) 
=
-qC_{b} W^{(q)}(b) + \beta Z^{(q)}(b)
	&=g(b)+1,
	\end{split}
\end{align}
where we define, for $b \geq 0$, 
\begin{align}
\begin{split}
g(b) &:= \rbra{1-\frac{rW^{(q)}(b)}{\Phi (q+r)Z^{(q)} (b , \Phi (q + r))}  }
\rbra{\beta Z^{(q)}(b)-1}- \frac{\beta q}{\Phi (q+r)} W^{(q)}(b) \\
&=
\frac{Z^{(q)\prime}(b, \Phi(q+r))}{\Phi (q+r)Z^{(q)} (b , \Phi (q + r))}  
\rbra{\beta Z^{(q)}(b)-1}- \frac{\beta q}{\Phi (q+r)} W^{(q)}(b).
\end{split}
  \label{g_0}
\end{align}
In other words, for $b > 0$, $v_b'(b) = 1$ if and only if $g(b) = 0$.


\begin{remark}[Probabilistic representation of $g$] \label{remark_g_probabilistic}

By \eqref{fpt_pr} and \eqref{g_0}, 
	\begin{align}\label{cond_opt_fpt}
	g(b)
	=\frac{q}{\Phi (q+r)}\frac{\beta \bE_0\sbra{e^{-q\tau^-_{-b}(r)}}-1}{Z^{(q)}(b)-\bE_0\sbra{e^{-q\tau^-_{-b}(r)}}}W^{(q)}(b). 
	\end{align} 
	\begin{enumerate}
	\item Because $Z^{(q)}(b)-\bE_0\sbra{e^{-q\tau^-_{-b}(r)}} > 0$ for $b > 0$ and $b \mapsto \beta \bE_0\sbra{e^{-q\tau^-_{-b}(r)}}-1$ is strictly decreasing, there exists at most one root of $g(b) = 0$.
	\item Using, in \eqref{cond_opt_fpt}, the fact that $\lim_{b\uparrow \infty} \bE_0\sbra{e^{-q\tau^-_{-b}(r)}}=0$, and $W^{(q)}(x)/Z^{(q)}(x) \xrightarrow{x \uparrow \infty} \Phi(q) / q$ as in Exercise 8.5 (i) in \cite{Kyp}, 
	we have that $\lim_{b\uparrow \infty} g(b)=- {\Phi(q)} /{\Phi(q+r)}<0$.
	Therefore $g(b)$ must be negative for sufficiently large $b$. 
	\end{enumerate}

\end{remark}

In order to handle also the case where such $b$ does not exist, we define
\begin{align}
b^\ast := \inf \cbra{ b \geq 0: 
g(b)  \leq 0}, \label{defbthreshold}
\end{align}
which is well-defined because, by Remark  \label{remark_g_probabilistic} (ii), the set $\cbra{ b \geq 0: 
g(b)  \leq 0} \neq \varnothing$.
Below, we provide a necessary and sufficient condition for the optimal barrier $b^*$ to be zero.

\begin{lemma} \label{lemma_criteria_zero}
	We have $b^* = 0$ if and only if $X$ is of bounded variation and
	\begin{align}
	\beta  -1  -\frac{r (\beta-1) + q \beta}{c\Phi(q+r)}  \leq 0. \label{criteria_b_0}
	\end{align}
\end{lemma}
\begin{proof}
	By the definition of $b^*$ as in \eqref{defbthreshold}, we have that $b^* = 0$ if and only if $g(0) \leq 0$ where, 
	by \eqref{g_0},
	\begin{align*}
	g(0)
	&= \beta  -1  -\left( r (\beta -1) + q \beta \right)  \frac{W^{(q)}(0)}{\Phi(q+r)}.
	\end{align*}
	For the case of unbounded variation (where $W^{(q)}(0) = 0$ by Remark \ref{remark_scale_function_properties}(2)), we have $g(0) = \beta - 1 > 0$ and hence $b^* > 0$. On the other hand, for the case of bounded variation, by Remark  \ref{remark_scale_function_properties}(2), $b^* = 0$ if and only if \eqref{criteria_b_0} holds.
%

\end{proof}
\begin{remark}[slope at $b^*$] \label{remark_slope_at_b}
(i) If $b^* > 0$ (i.e.\ $g(b^*)=0$), the equation \eqref{slope_b} implies $v_{b^*}'(b^*)=1$. (ii) If $b^*=0$ (i.e.\ $g(0)\leq0$), \eqref{slope_b} gives $v_{b^*}'(0+)\leq 1$.
\end{remark}

\begin{remark} \label{C_b_simplification} Suppose $b^* > 0$ (i.e. $g(b^*) = 0$).  
Then by \eqref{slope_b}, we have
\begin{align*}
C_{b^*} = \frac {\beta Z^{(q)}(b^*) - 1}{q W^{(q)}(b^*)}. 
\end{align*}
\end{remark}

\section{Verification of optimality}\label{VoO} 
In this section, we will show the optimality of the strategy $\bar{\pi}^{0,b^*}$ for the value of $b^*$ selected in the previous section.


\begin{theorem} \label{main_theorem}The strategy $\bar{\pi}^{0,b^*}$ is optimal and the value function of the problem $(\ref{control:value})$ is given by $v = v_{b^*}$.
\end{theorem}

In order to show Theorem \ref{main_theorem}, it suffices to show variational inequalities. We omit the proof of the following proposition because
 it is essentially the same as the spectrally positive case given in Lemma 5.3 of \cite{PerYam}.
Here we slightly relax the assumption on the smoothness at zero, which can be done by applying the Meyer-It\^o formula as in Theorem 4.71 of \cite{Pro}. 

Let $\mathcal{L}$ be the infinitesimal generator associated with
the process $X$ applied to a measurable function $f$ on $\R$ that is $C^1 (0, \infty)$ (resp.\ $C^2 (0, \infty)$) for the case in which $X$ is of bounded (resp.\ unbounded) variation with
\begin{align*} 
\mathcal{L} f(x) := \gamma f'(x) + \frac 1 2 \sigma^2 f''(x) + \int_{(-\infty,0)} \left[ f(x+z) - f(x) -  f'(x) z 1_{\{-1 < z < 0\}} \right] \Pi(\diff z). 
\end{align*}
Below, as in Avram et al.\ \cite{AvrPalPis}, we extend the domain of $v_{\pi}$ of \eqref{v_pi} to $\R$ by setting $v_{\pi}(x) = \beta x + v_{\pi}(0)$ for $x < 0$.
\begin{proposition}\label{verificationlemma}
Suppose $\hat{\pi} \in \mathcal{A}$ 
such that $v_{\hat{\pi}}$ is $C^1 (0, \infty)$ (resp.\ $C^2 (0, \infty)$) for the case $X$ is unbounded (resp.\ unbounded) variation, continuous on $\R$ and, for the case of unbounded variation, continuously differentiable at zero. In addition, suppose that 
\begin{align}\label{v1}
\begin{split}
({\cal{L}}-q)v_{\hat{\pi}} (x) +r \max_{0 \leq l \leq x} 
\{l+ v_{\hat{\pi}} (x - l)- v_{\hat{\pi}} (x)\}  \leq 0 ,& ~~~x > 0, 
\\
v_{\hat{\pi}}^\prime (x) \leq \beta ,& ~~~x>0, \\
\inf_{x \geq 0}v_{\hat{\pi}}(x)> -m,& ~~~\text{for some } m>0. 
\end{split}
\end{align}  
Then $v_{\hat{\pi}} (x)=v(x)$ for all $x \geq 0$ and hence $\hat{\pi}$ is an optimal strategy. 
\end{proposition}
We shall provide some preliminary results in order to show the variational inequalities.
\begin{lemma}\label{ineq_gen}
For $b\geq 0$, we have
\begin{align}
({\cal{L}}-q)v_{b}(x)=
\begin{cases}
0   &\text{if} ~x\in (0 , b), \\
-r \{ (x - b) + v_{b} (b)-v_{b}(x) \}
~~&\text{if}~ x \in[b,  \infty).
\end{cases}
\label{405}
\end{align}
\end{lemma}
\begin{proof}
(i) Suppose $0 < x < b$. 
By the proof of Theorem 2.1 in \cite{BKY}, we have 
\begin{align}
(\mathcal{L}-q) Z^{(q)}(y) &=(\mathcal{L}-q) \Big(\overline{Z}^{(q)}(y) + \frac {\psi'(0+)} q \Big) = 0, \quad y > 0. \label{martingale_Z_R}
\end{align}
Applying these in \eqref{vf_new_label_x<b}, we obtain \eqref{405}. \par
(ii) Suppose $x > b$. 
By the proof of Lemma 5.2 in \cite{NobPerYamYan},
we have
\begin{align}
({\cal{L}}-q)\bar{W}^{(q+r)}(x - b)
&=1+r \bar{W}^{(q+r)} (x - b), \notag\\
({\cal{L}}-q)\bar{\bar{W}}^{(q+r)} (x - b)
&=(x -b) +r\bar{\bar{W}}^{(q+r)}(x - b). \label{411}
\end{align}
On the other hand by the proof of Lemma 4.5 in \cite{EgaYam}, we have
\begin{multline*}
({\cal{L}}-(q+r))\rbra{\int_0^{x-b} W^{(q+r)}(y)Z^{(q)}(x-y)\diff y} \n \\
=({\cal{L}}-(q+r))\rbra{\int_0^{x-b} 
	W^{(q+r)}(x-b-y)Z^{(q)}(b + y)\diff y} =Z^{(q)}(x),
\end{multline*} 
and hence
\begin{align}
({\cal{L}}-q)\rbra{\int_0^{x-b} W^{(q+r)}(y)Z^{(q)}(x-y)\diff y}
= Z^{(q,r)}_{-b}(x-b).
 \label{416}
\end{align}
Combining \eqref{martingale_Z_R} and \eqref{416}, we obtain 
\begin{equation}\label{418_a}
({\cal{L}}-q)Z^{(q, r)}_{-b} (x - b)=rZ^{(q, r)}_{-b} (x - b).
\end{equation}
In a similar way, we see that 
\begin{align}
({\cal{L}}-q)\rbra{\int_0^{x-b} W^{(q+r)}(y)\bar{Z}^{(q)} (x-y)\diff y}=
\overline{Z}^{(q,r)}_{-b}(x-b),
\label{417}
\end{align}
and, using identities
 \eqref{martingale_Z_R} and \eqref{417}, we obtain 
\begin{align}
({\cal{L}}-q)\Big(\bar{Z}^{(q,r)}_{-b} (x-b)+\frac{\psi'(0+)}{q}&
\Big)
=r \overline{Z}^{(q,r)}_{-b}(x-b).
\label{418}
\end{align}
Therefore, applying \eqref{411}, \eqref{418_a} and \eqref{418} in \eqref{vf_new_label}, 
\begin{align}
({\cal{L}}-q)v_{b}(x)&= -C_{b}\left(rZ^{(q, r)}_{-b} (x - b)
-r Z^{(q)}(b)\left(1+r\overline{W}^{(q+r)}(x-b)\right)\right)-r\left((x-b)+r\bar{\bar{W}}^{(q+r)}(x - b)\right)\notag\\
&-r\beta \overline{Z}^{(q)}(b)\left(1+r\overline{W}^{(q+r)} (x - b)\right)+r\beta \overline{Z}^{(q,r)}_{-b}(x-b) \notag \\
&=-r\rbra{(x - b) + v_{b} (b)-v_{b}(x)},\notag
\end{align}
where in the last equality we used the fact that $v_{b}(b)= -C_{b}Z^{(q)}(b)+\displaystyle\beta \Big( \bar{Z}^{(q)} (b)+\frac{\psi^\prime(0+)}{ q}\Big)$.
\end{proof}
\begin{lemma}\label{cond_b^*}
We
 have $1\leq v_{b^\ast}^\prime (x) \leq \beta$ for 
$x \in (0 , b^\ast)$ and $0\leq v_{b^\ast}^\prime (x)\leq 1$ for $x\in (b^\ast , \infty)$. 
\end{lemma}
\begin{proof}
We prove separately for the cases (i) $b^* > 0$ and (ii) $b^* = 0$.

(i) Suppose $b^*>0$. 
Then, using \eqref{302a} and  
Remark \ref{C_b_simplification},
we obtain
\begin{align}\label{431}
v'_{b^*}(x)
&=\beta Z^{(q,r)}_{-b^*}(x-b^*) -r\overline{W}^{(q+r)}(x-b^\ast)-\frac{\beta Z^{(q)}(b^*)-1}{W^{(q)}(b^*)} W^{(q,r)}_{-b^*}(x-b^*).
\end{align}
By the second equality of \eqref{g_0} and the fact that $g(b^*)=0$, we obtain
\begin{align*}
q\frac{Z^{(q)}(b^*,\Phi(q+r))}{Z^{(q)\prime}(b^*,\Phi(q+r))}W^{(q)}(b^*)=\frac{\beta Z^{(q)}(b^*)-1}{\beta }.
\end{align*}
Hence using the above expression and \eqref{431}  in \eqref{403} we obtain that, for $x > 0$, 
\begin{align}\label{404}
\begin{split}
\beta\mathbb{E}_{x-b^*}\Big[&e^{-q\tau^-_{-b^*}(r)}\Big]=\beta Z_{-b^*}^{(q,r)}(x-b^*)-r\beta Z^{(q)}(b^*)\overline{W}^{(q+r)}(x-b^*) \\&-\frac{\beta Z^{(q)}(b^*)-1}{W^{(q)}(b^*)} \left(W^{(q,r)}_{-b^*}(x-b^*)-rW^{(q)}(b^*) \overline{W}^{(q+r)}(x-b^*)\right) \\
&=v'_{b^*}(x),
\end{split}
\end{align}
where the last inequality follows from \eqref{431}. 



From \eqref{404},  we then deduce that $0\leq v'_{b^*}(x)\leq \beta$ and that $v_{b^*}'$ is decreasing on $(0,\infty)$. This and the fact that $v'_{b^*}(b^*)=1$ as in Remark \ref{slope_b} complete the proof.

(ii) Suppose that $b^*=0$ (then necessarily $X$ is of bounded variation by Lemma \ref{lemma_criteria_zero}). 
Because
\begin{align*}
C_0 = \frac {r(\beta-1) + q \beta} {q \Phi(q+r)},
\end{align*}
we have, by \eqref{302a} and Remark \ref{identities_LRZ},
		\begin{align}\label{der_0}
		v_0^{\prime} (x) &= - \frac {r(\beta-1) + q \beta} {\Phi(q+r)} W^{(q,r)}_0(x)-r\overline{W}^{(q+r)}(x)+\beta Z^{(q,r)}_{0}(x) \notag\\
&=
\frac {r(\beta-1) + q \beta} {r+q}
\left( Z^{(q+r)}(x)-\frac{r+q}{\Phi(q+r)}W^{(q+r)}(x)\right)+\frac{r}{r+q}.
\end{align}

Differentiating  \eqref{der_0} further, 
\begin{align}\label{sec_der_0}
v_0''(x+)=
(r(\beta-1) + q \beta)
\left(1-\frac{1}{\Phi(q+r)}\frac{W^{(q+r)\prime}(x+)}{W^{(q+r)}(x)}\right)W^{(q+r)}(x).
\end{align}
Because $\beta > 1$, we have $r(\beta-1) + q \beta  > 0$. In addition, $x \mapsto W^{(q+r)\prime}(x+) /W^{(q+r)}(x) $ is monotonically decreasing in $x$ as in (8.18) and Lemma 8.2 of \cite{Kyp} and converges to $\Phi(r+q)$ as $x \rightarrow \infty$.  By these facts and \eqref{sec_der_0}, we have that  $v_0''(x+) < 0$, meaning $v_0$ is concave.
\par Recall, as in Remark \ref{remark_slope_at_b}, that $v_0'(0+) \leq 1$. Hence we have that $v_0'(x)\leq 1$ for all $x\in(0,\infty)$. Finally, we have $v_0'(x) \xrightarrow{x \uparrow \infty} r/(r+q) > 0$ because $ Z^{(q+r)}(x) -  {(r+q)}  W^{(q+r)}(x) /  {\Phi(r+q)}  $ vanishes in the limit by Theorem 8.1 (ii) of \cite{Kyp}, and hence $v'_0(x) > 0$.
\end{proof}
Next, 
by applying
 Lemma \ref{cond_b^*} for $b^* > 0$ and $b^* = 0$, the following results are immediate.
\begin{lemma}\label{cond_max_v}
	For $b^*\geq0$, we have that
	\begin{equation*}
	\max_{0\leq l\leq x} \{ l+v_{b^*}(x-l)-v_{b^*}(x) \} =
	\begin{cases} 0 &\mbox{if } x \in[0,b^*], \\ 
	x-b^*+v_{b^*}(b^*)-v_{b^*}(x) & \mbox{if } x \in (b^*,\infty). \end{cases}
	\end{equation*}
\end{lemma}
We are now ready to show Theorem \ref{main_theorem}.
\begin{proof}[Proof of Theorem \ref{main_theorem}]
We shall show that $v_{b^*}$ satisfies all the conditions given in Proposition  \ref{verificationlemma}.
The desired smoothness/continuity of $v_{b^*}$ holds by Lemma \ref{lemma_smooth_fit} and Remark \ref{remark_smoothness_zero}. Hence, we only need to prove the variational inequalities given in \eqref{v1}.
\par Lemmas \ref{ineq_gen} and \ref{cond_max_v} yield the validity of the first item of \eqref{v1} with equality. The second item holds by Lemma \ref{cond_b^*}. 
Finally, the third item follows because, by the monotonicity of $v_{b^*}$ in view of Lemma \ref{cond_b^*} and \eqref{mean-finite},  we have $\inf_{x\geq0}v_{b^*}(x) \geq  v_{b^*}(0) > -\infty$.
\end{proof}

\section{Numerical results} \label{section_numerics}

In this section, we give numerical results using the  spectrally negative \lev process with phase-type jumps of the form:
\begin{equation*}
 X(t) - X(0)= c t+\eta B(t) - \sum_{n=1}^{N(t)} Z_n, \quad 0\le t <\infty, 
\end{equation*}
where $B=( B(t); t\ge 0)$ is a standard Brownian motion, $N=(N(t); t\ge 0 )$ is a Poisson process with arrival rate $1$, and  $Z = ( Z_n; n = 1,2,\ldots )$ is an i.i.d.\ sequence of phase-type random variables that approximate the (folded) Normal distribution with mean zero and variance $1$ (the phase-type parameters are given in \cite{leung2015analytic}). See, e.g., \cite{Maier} for a review of the phase-type distribution. The processes $B$, $N$, and $Z$ are assumed to be mutually independent.  We refer the reader to \cite{Egami_Yamazaki_2010_2, KKR} for the forms of the corresponding scale functions. Throughout we set $q = 0.05$.


 \begin{figure}[htbp]
\begin{center}
\begin{minipage}{1.0\textwidth}
\centering
\begin{tabular}{cc}
 \includegraphics[scale=0.35]{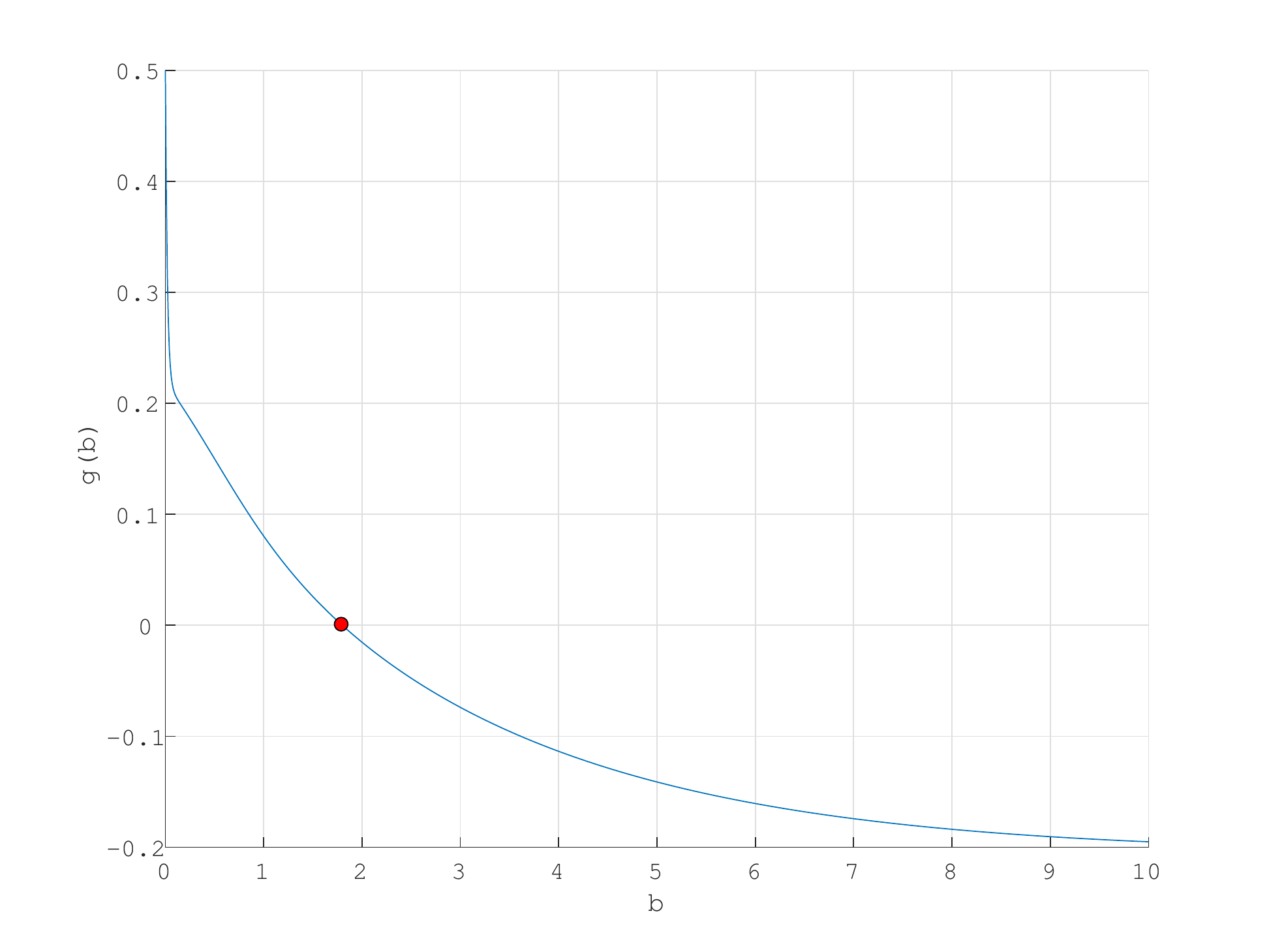} & \includegraphics[scale=0.35]{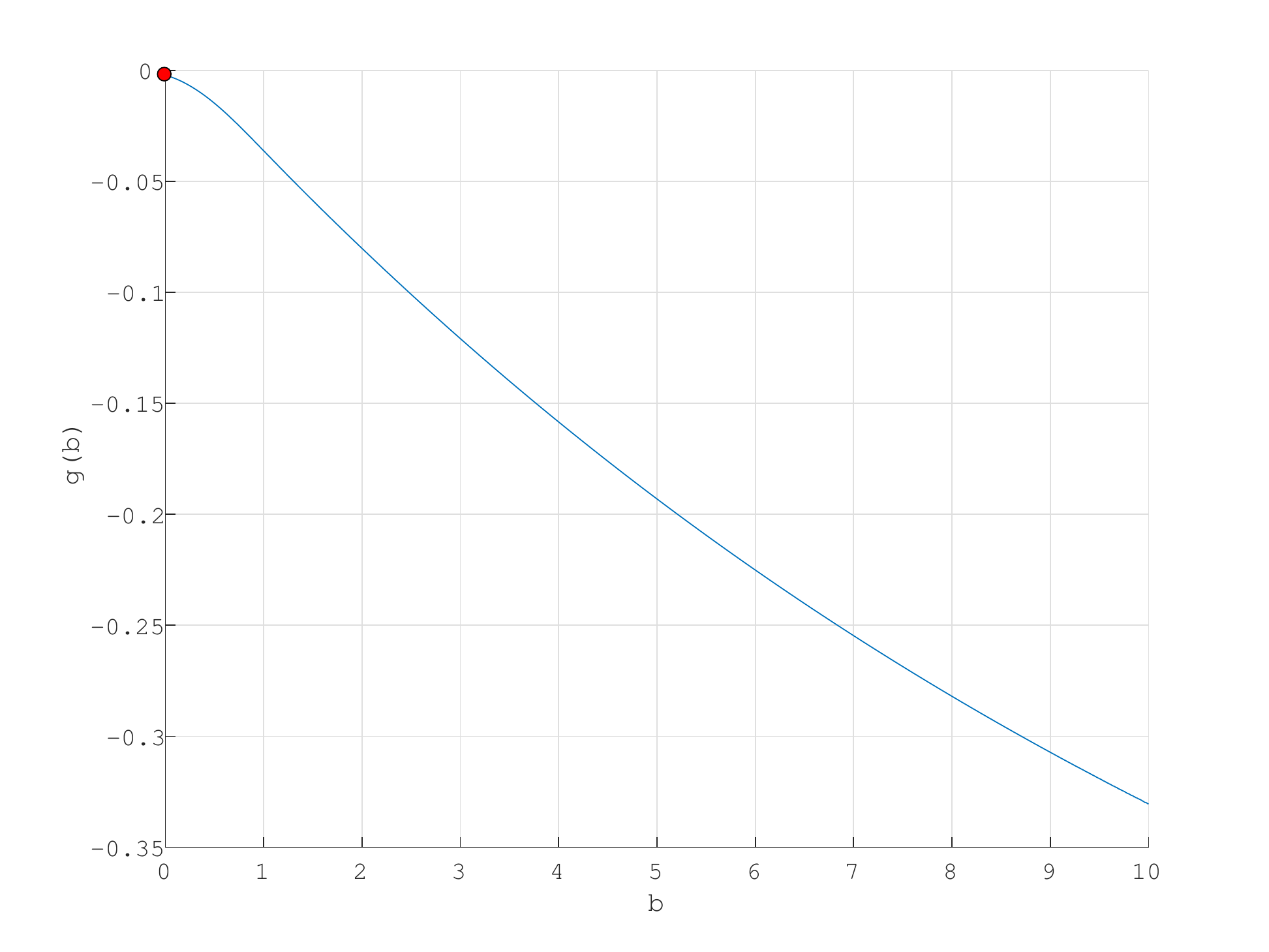}  \\
 \textbf{Case} 1 & \textbf{Case} 2 \end{tabular}
\end{minipage}
\caption{Plots of $b \mapsto g(b)$  for \textbf{Cases} 1 and 2.  The values of $b^*$ are indicated by the circles.
} \label{figure_g}
\end{center}
\end{figure}

 \begin{figure}[htbp]
\begin{center}
\begin{minipage}{1.0\textwidth}
\centering
\begin{tabular}{cc}
 \includegraphics[scale=0.35]{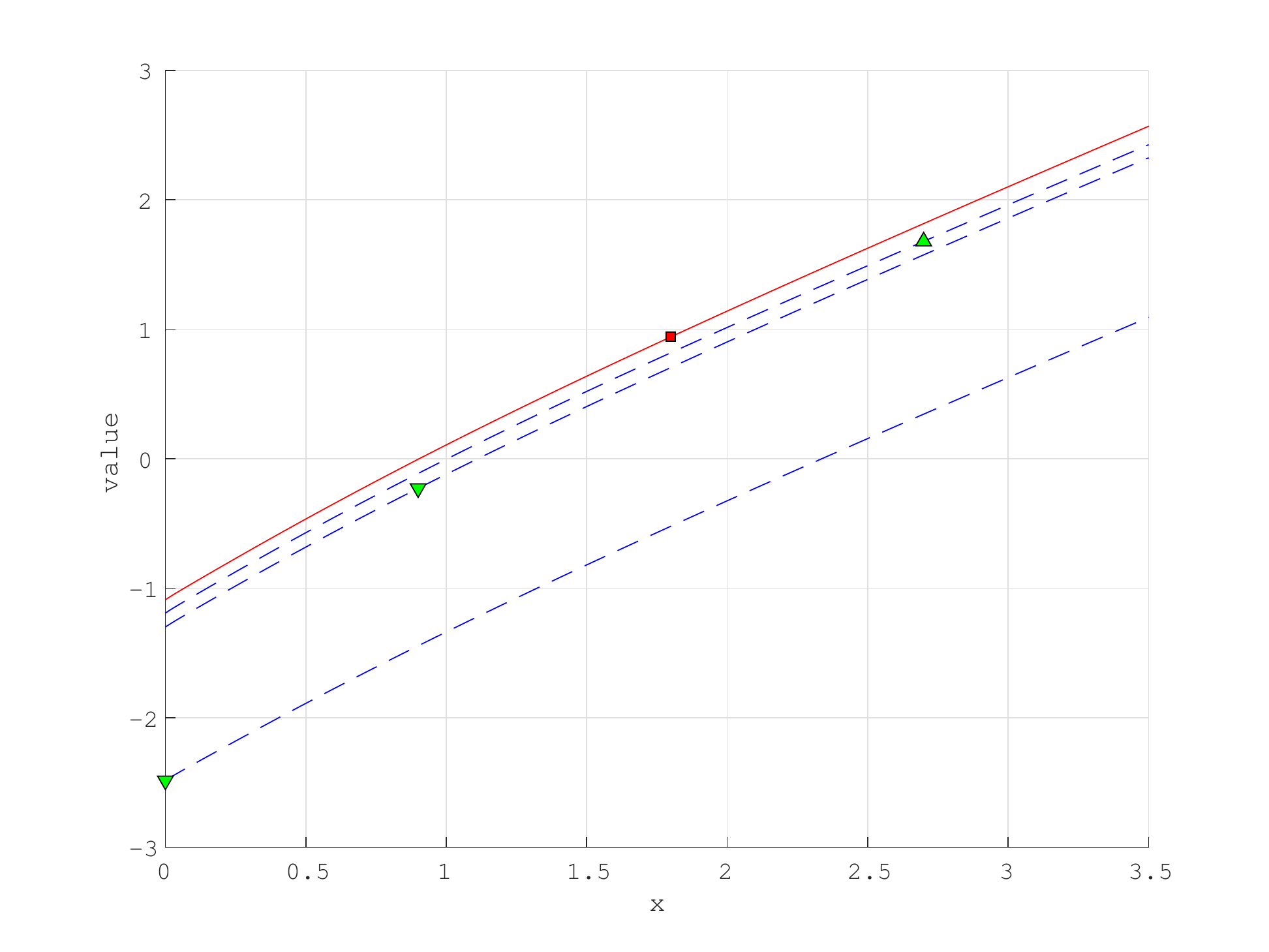} & \includegraphics[scale=0.35]{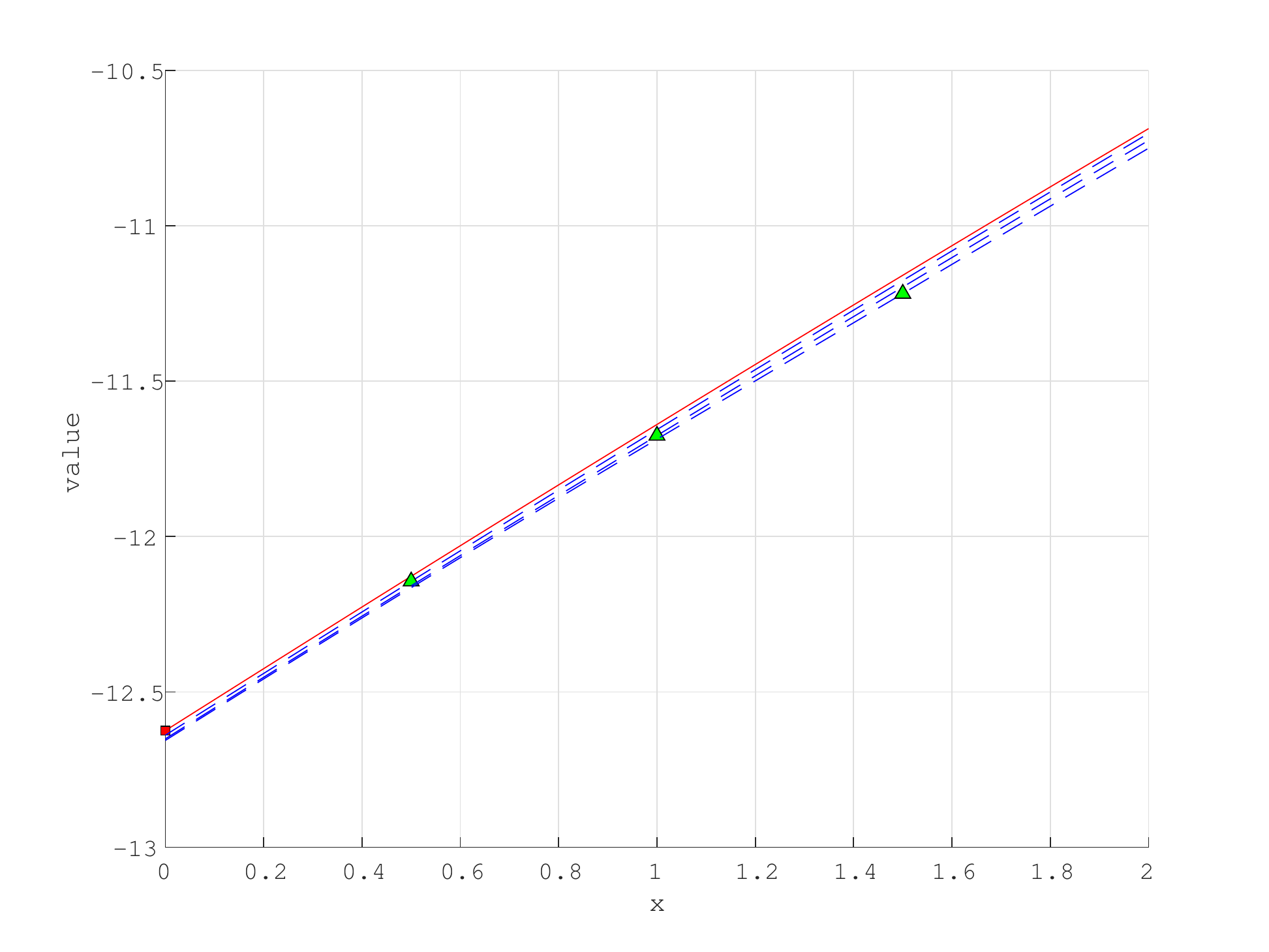}  \\
 \textbf{Case} 1& \textbf{Case} 2 \end{tabular}
\end{minipage}
\caption{Plots of $v_{b^*}$ (solid)   for \textbf{Cases} 1 and 2 in comparison to $v_b$ (dotted) with $b = 0, b^*/2, 3b^*/2$ for \textbf{Case} 1 and  $b = 0.5, 1,1.5$ for \textbf{Case} 2.
The points $(b^*, v_{b^*}(b^*))$ are indicated by the squares and the points $(b, v_{b}(b))$ are  indicated by the down-pointing (resp.\ up-pointing) triangles  if $b < b^*$ (resp.\ $b > b^*$).
} \label{figure_values}
\end{center}
\end{figure}

We first illustrate the implementation procedure using  \textbf{Case} 1 (unbounded variation) with $\eta = 0.2$, $c =1$ and $\beta = 1.5$ and \textbf{Case} 2  (bounded variation) with $\eta = 0$, $c = 0.3$ and $\beta = 1.05$ with the common value of $r = 0.5$.

Recall the definition of $b^*$ as in \eqref{defbthreshold}.  In Figure \ref{figure_g}, the function $g(b)$ as in \eqref{g_0} is plotted for \textbf{Cases} 1 and 2. As we have studied in Remark  \ref{remark_g_probabilistic} and  Lemma \ref{lemma_criteria_zero}, if $g(0) > 0$ as in \textbf{Case} 1, there exists a unique value $b^*$ such that $g(b^*) = 0$, and hence this can be computed by the bisection method.  For the case $g(0) \leq  0$ as in \textbf{Case} 2, we set $b^* = 0$.
Using the selected $b^*$, the optimal value functions $v_{b^*}$ are computed and plotted in Figure \ref{figure_values} for both \textbf{Cases} 1 and 2.  In the same plots, in order to confirm the optimality, we plot the function $v_b$ for different selection of $b$.  It is confirmed that $v_{b^*}$ dominates $v_b$, for $b \neq b^*$, uniformly in $x$.


Figure \ref{fig_sensitivity} shows the behaviors of the optimal solutions with respect to the parameters $\beta$ and $r$ using the same parameters as \textbf{Case} 1 (unless stated otherwise for the values of $\beta$ and $r$).  The left plot shows  $v_{b^*}$ for $\beta$ ranging from $1.01$ to $20$.
As expected, $v_{b^*}$ is decreasing in $\beta$ uniformly in $x$.  In addition, we observe that $b^*$ increases as $\beta$ increases.  The right plot shows $v_{b^*}$ for various values of $r$ ranging from $0.0001$ to $50$ along with the results in the classical bail-out case (without the restriction \eqref{restriction_poisson}), say $\bar{v}_{b^\dagger}$ with the optimal classical barrier $b^\dagger$, as in \cite{AvrPalPis}. It is observed that the value function converges increasingly to that in \cite{AvrPalPis}.  It is also confirmed that $b^*$ increases in $r$ and converges to $b^\dagger$ of \cite{AvrPalPis} as $r \rightarrow \infty$.
 \begin{figure}[htbp]
\begin{center}
\begin{minipage}{1.0\textwidth}
\centering
\begin{tabular}{cc}
 \includegraphics[scale=0.35]{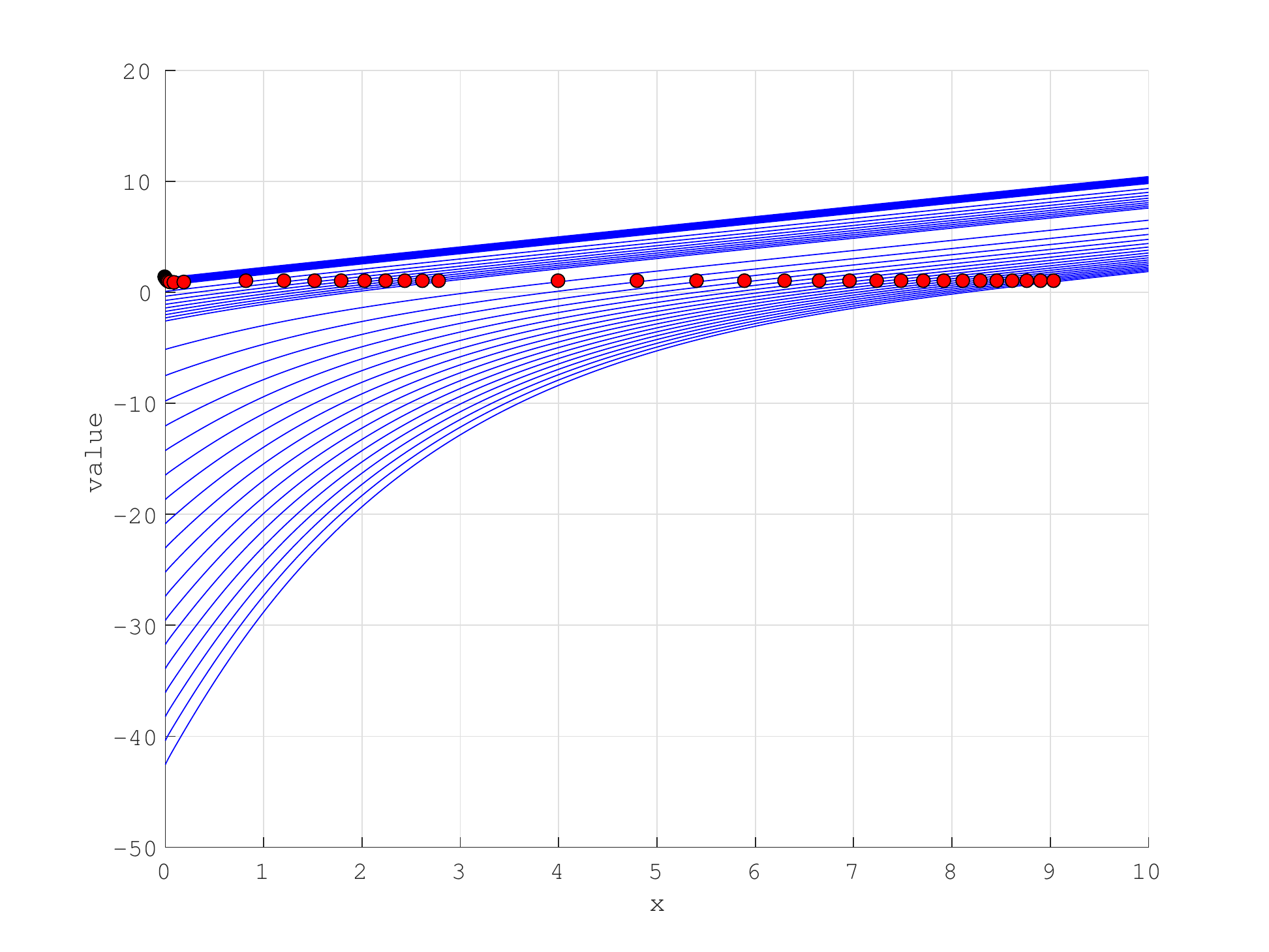} & \includegraphics[scale=0.35]{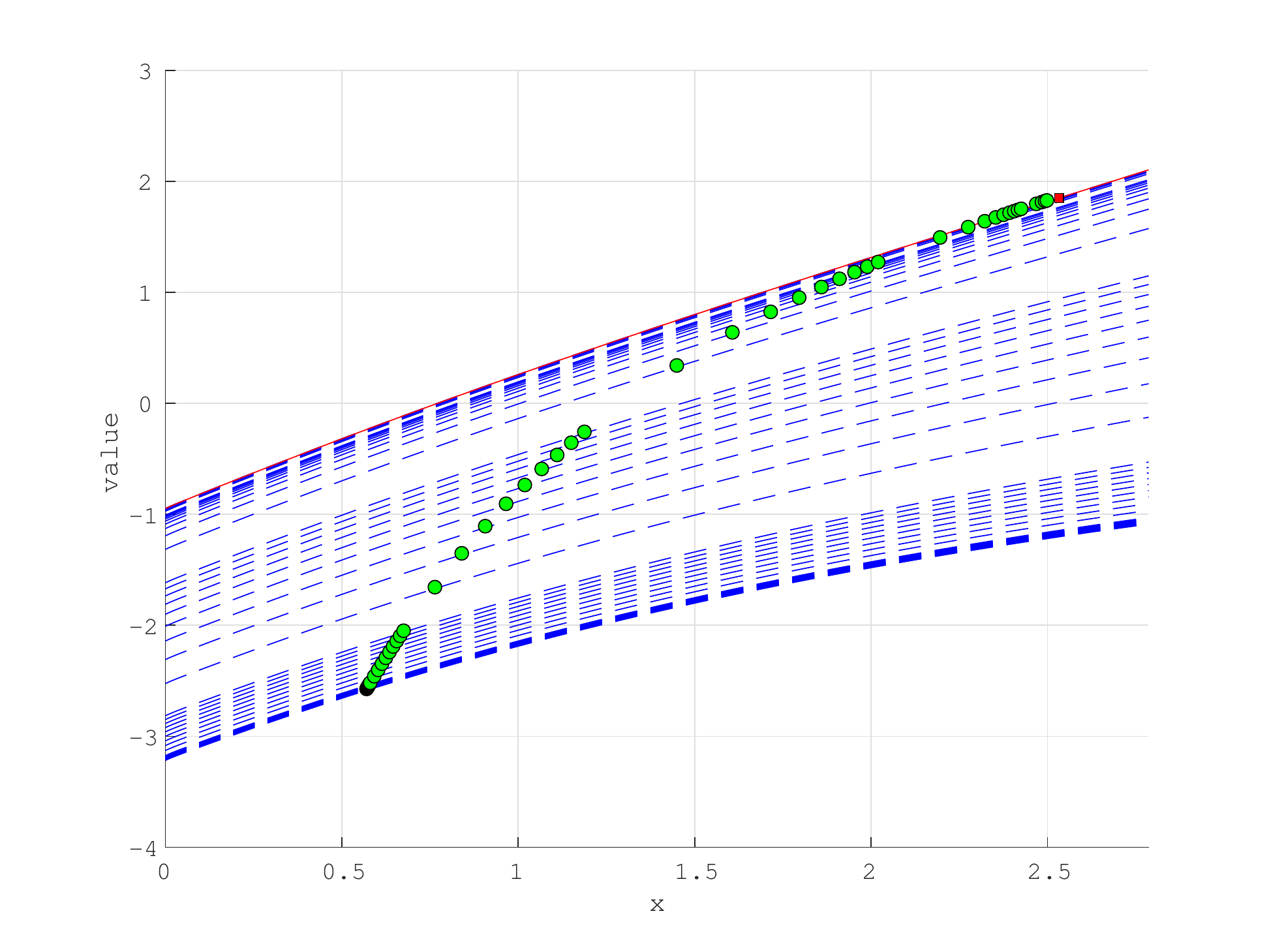}  \\
 with respect to $\beta$ & with respect to $r$ \end{tabular}
\end{minipage}
\caption{(Left) Plots of $v_{b^*}$ for $\beta=1.01$, $1.02$, $\ldots$, $1.09$, $1$, $1.1$, $1.2$, $\ldots$, $1.9$, $2$, $3$, $\ldots$, $19$, $20$ with the points $(b^*, v_{b^*}(b^*))$ indicated by the circles. (Right) Plots of $v_{b^*}$ (dotted) for $r=0.0001$, $0.0002$, $\ldots$, $0.0009$, $0.001$, $0.002$, $\ldots$, $0.009$, $0.01$, $0.02$, $\ldots$, $0.1$, $0.2$, $\ldots$, $0.9$, $1$, $2$, \ldots $9$, $10$, $20$, $30$, $40$, $50$ with the points $(b^*, v_{b^*}(b^*))$ indicated by circles, along with the classical case $\bar{v}_{b^\dagger}$ as in \cite{AvrPalPis} with the point $(b^\dagger, \bar{v}_{b^\dagger}(b^\dagger))$ indicated by the square.
} \label{fig_sensitivity}
\end{center}
\end{figure}
\appendix

	\section*{Acknowledgements}

K.\ Noba and K.\ Yano were supported by JSPS-MAEDI Sakura program.
K. Yano was supported by MEXT KAKENHI grant no.'s 26800058, 15H03624
and 16KT0020.
J. L. P\'erez  was  supported  by  CONACYT,  project  no.\ 241195.
K. Yamazaki was supported by MEXT KAKENHI grant no.\  	17K05377.

\end{document}